\documentclass[11pt]{amsart}
\usepackage[utf8]{inputenc}
\usepackage{enumerate,amssymb,xcolor,comment}
\usepackage[a4paper, footskip=1cm, headheight = 16pt, top=3.7cm, bottom=3cm,  right=2.5cm,  left=2.5cm ]{geometry}
\usepackage[colorlinks,linkcolor=blue,citecolor=red,hypertexnames=false]{hyperref}
\usepackage{thm-restate}

\newtheorem{question}{Question}[section]
\newtheorem{lemma}[question]{Lemma}
\newtheorem{theorem}[question]{Theorem}
\newtheorem{conjecture}[question]{Conjecture}

\newtheorem{remark}[question]{Remark}

\usepackage[foot]{amsaddr}

\usepackage[T1]{fontenc}

\usepackage[leqno]{amsmath}

\makeatletter
\newcommand{\leqnomode}{\tagsleft@true}
\newcommand{\reqnomode}{\tagsleft@false}
\makeatother

\usepackage{latexsym}
\usepackage{amsfonts}

\usepackage{tikz}
\usepackage{float}
\usepackage{lmodern}

\def\dd{\hbox{-}}

\usepackage{marvosym}

\DeclareMathOperator{\tw}{tw}

\DeclareMathOperator{\Hub}{Hub}
\DeclareMathOperator{\Core}{Core}
\DeclareMathOperator{\ta}{tree-\alpha}
\DeclareMathOperator{\tc}{tree-\overline{\chi}}

\newcounter{tbox}

\newcommand{\sta}[1]{\vspace*{0.3cm}\refstepcounter{tbox}\noindent{ \parbox{\textwidth}{(\thetbox) \emph{#1}}}\vspace*{0.3cm}}

\newcommand{\mylongtitle}[1]{%
  \ifodd\value{page}%
    \protect\parbox{0.97\linewidth}{#1}\hfill%
  \else%
    \hfill\protect\parbox{0.97\linewidth}{#1}%
  \fi%
}

\renewcommand{\S}{\mathcal{S}}
\newcommand{\CC}{\mathcal{C}_{\text{nc}}}
\newcommand{\cchi}{\overline{\chi}}

\makeatletter
\newcommand{\otherlabel}[2]{\protected@edef\@currentlabel{#2}\label{#1}}
\makeatother

\mathchardef\mh="2D

\title[Tree independence number I. (Even hole, diamond, pyramid)-free graphs]{Tree independence number I. (Even hole, diamond, pyramid)-free graphs}

\author{Tara Abrishami$^{\ast \dagger}$}
\author{Bogdan Alecu$^{\ast \ast \mathparagraph}$}
\author{Maria Chudnovsky$^{\ast \ast \ast \amalg}$}
\author{Sepehr Hajebi$^{\mathsection}$}
\author{Sophie Spirkl$^{\mathsection \parallel}$}
\author{Kristina Vu\v{s}kovi\'{c}$^{\ast \ast \ddagger}$}

\address{$^{\ast}$Department of Mathematics, University of Hamburg, Germany}
\address{$^{\ast \ast}$School of Computing, University of Leeds, Leeds, UK}
\address{$^{\ast \ast \ast}$Princeton University, Princeton, NJ, USA}
\address{$^{\mathsection}$Department of Combinatorics and Optimization, University of Waterloo, Waterloo, Ontario, Canada}
\address{$^{\dagger}$ Supported by the National Science Foundation Award Number DMS-2303251 and the Alexander von Humboldt Foundation. This work was performed while the author was at Princeton University and supported by NSF-EPSRC Grant DMS-2120644.}
\address{$^{\amalg}$ Supported by NSF-EPSRC Grant DMS-2120644 and by AFOSR grant FA9550-22-1-0083.} 
     \address{$^{\mathparagraph}$ Supported by DMS-EPSRC Grant EP/V002813/1.} 
\address{$^{\parallel}$ We acknowledge the support of the Natural Sciences and Engineering Research Council of Canada (NSERC), [funding reference number RGPIN-2020-03912].
Cette recherche a \'et\'e financ\'ee par le Conseil de recherches en sciences naturelles et en g\'enie du Canada (CRSNG), [num\'ero de r\'ef\'erence RGPIN-2020-03912]. This project was funded in part by the Government of Ontario.}
\address{$^{\ddagger}$ Partially supported by DMS-EPSRC Grant EP/V002813/1.}
\address{Corresponding author email: b.alecu@leeds.ac.uk}

\date {\today}

\allowdisplaybreaks

\begin{document}

\maketitle

\begin{abstract}
    The tree-independence number $\ta$, first defined and studied by Dallard, Milani\v{c} and \v{S}torgel, is a variant of treewidth tailored to solving the maximum independent set problem. 

    Over a series of papers, Abrishami et al.\ developed the so-called central bag method to study induced obstructions to bounded treewidth. Among others, they showed that, in a certain superclass $\mathcal C$ of (even hole, diamond, pyramid)-free graphs, treewidth is bounded by a function of the clique number. In this paper, we relax the bounded clique number assumption, and show that $\mathcal C$ has bounded $\ta$. Via existing results, this yields a polynomial-time algorithm for the Maximum Weight Independent Set problem in this class. Our result also corroborates, for this class of graphs, a conjecture of Dallard, Milani\v{c} and \v{S}torgel that in a hereditary graph class, $\ta$ is bounded if and only if the treewidth is bounded by a function of the clique number. 
    
\end{abstract}

\section{Introduction}
All graphs in this paper are finite and simple. Let $G = (V(G), E(G))$ be a graph. An {\em induced subgraph} of $G$ is a subgraph of $G$ obtained by deleting vertices. More explicitly, if $X \subseteq V(G)$, then $G[X]$ denotes the subgraph of $G$ induced by $X$, with vertex set $X$ and edge set $E(G) \cap {X \choose 2}$. In this paper, we use induced subgraphs and their vertex sets interchangeably. 
For a graph $H$, we say that $G$ {\em contains} $H$ if $H$ is isomorphic to an induced subgraph of $G$, and
$G$
is {\em $H$-free} if it does not contain $H$. If $\mathcal{H}$ is a set of graphs, then we say that $G$ is  {\em $\mathcal{H}$-free} 
if $G$ is $H$-free for every $H \in \mathcal{H}$. 

A {\em tree decomposition $(T, \varphi)$ of $G$} consists of a tree $T$ and a map $\varphi: V(T) \to 2^{V(G)}$ satisfying the following: 
\begin{enumerate}[(i)]
\item For all $v \in V(G)$, there exists $t \in V(T)$ such that $v \in \varphi(t)$; 

\item For all $v_1v_2 \in E(G)$, there exists $t \in V(T)$ such that $v_1, v_2 \in \varphi(t)$; 
\item For all $v \in V(G)$, the subgraph of $T$ induced by $\{t \in V(T) \text{ s.t. } v \in \varphi(t)\}$ is connected. 
\end{enumerate}
The {\em width} of a tree decomposition $(T,\varphi)$ is $\max_{t\in V(T)}|\varphi (t)|-1$.
The {\em treewidth} of $G$, denoted by $\tw (G)$, is the minimum width of a tree decomposition of $G$.

A {\em stable set} in a graph $G$ is a set of pairwise non-adjacent vertices of $G$. The {\em independence number} $\alpha(G)$ of $G$ is the size of a maximum stable set in $G$. The {\em independence number} of a tree decomposition 
$(T, \varphi)$ of $G$ is $\max_{t \in V(T)} \alpha(G[\varphi(t)])$. The {\em tree independence number} of $G$, denoted $\ta(G)$, is the minimum independence number of a tree decomposition of $G$.

 The tree independence number was defined and studied by Dallard, Milani\v{c} and \v{S}torgel in \cite{dms2}, in the context of studying the complexity of the 
Maximum Weight Independent Set (MWIS)  problem on graph classes whose treewidth 
is large only due to the presence of a large clique. It is shown in \cite{dms2} that if a graph is given together with a tree decomposition with bounded independence number, then the MWIS problem 
can be solved in polynomial time. In \cite{dfgkm}, it is then shown how to compute such tree decompositions efficiently in graphs of bounded $\ta$, yielding an efficient algorithm for the MWIS problem for graphs of bounded $\ta$. 

In \cite{dms1}, a graph class $\mathcal{G}$ is called {\em $(\tw,\omega )$-bounded} if there  exists a function 
$f: \mathbb{N}\to \mathbb{N}$ such that for every graph $G\in \mathcal{G}$, and for every induced subgraph $H$ of $G$, $\tw (H)\leq f(\omega(H))$ (where $\omega(G)$ is the size of a largest clique in $G$).
Such a function $f$ is called a 
{\em $(\tw,\omega)$-bounding function} for the class $\mathcal{G}$.
By Ramsey's theorem, graph classes of bounded tree independence number are $(\tw,\omega )$-bounded.
Furthermore, all of the following graph classes have bounded tree independence number: graph classes of bounded treewidth, 
intersection graphs of connected subgraphs of graphs with treewidth at most $t$, for any fixed positive integer $t$
(these contain for example chordal graphs and circular arc graphs), and
graph classes in which the size of a minimal separator is bounded. More examples and further discussion of these parameters is given in \cite{dms1, dms2, dms3}.
The question remains whether the $(\tw,\omega )$-bounded property has algorithmic implications for the MWIS problem. This question would be settled if the following conjecture from \cite{dms2} holds true:

\begin{conjecture}[\cite{dms2}]\label{con:tw-omega-implies-ta}
    A class $\mathcal C$ of graphs is $(\tw, \omega)$-bounded if and only if it has bounded $\ta$.
\end{conjecture}

Before we state our main result, we define several types of graphs.
A {\em clique} in a graph is a set of pairwise adjacent vertices, and for an integer $t\geq 1$, $K_t$ denotes the {\em complete graph} on $t$ vertices. 
A {\em path} is a tree of maximum degree at most 2. A {\em path in $G$} is an induced subgraph of $G$ that is a path. 
A {\em diamond} is the graph obtained by deleting an edge from $K_4$. 
A {\em hole} of a graph $G$ is an induced cycle of length at least four. By $C_4$ we denote the hole of length four. 
A {\em wheel} $(H, w)$ is a hole $H$ and a vertex $w \in V(G)$ such that $w$ has at least three  neighbors in $H$. An {\em even wheel} is a wheel $(H, w)$ such that $w$ has an even number of neighbors in $H$.
A {\em theta} is a graph consisting of two vertices $a$ and $b$ and three distinct paths $P_1$, $P_2$, $P_3$ from
$a$ to $b$ such that any two of them induce a hole (so in particular $a$ and $b$ are nonadjacent, and hence all three paths are of length at least 2).
A {\em pyramid} is a graph consisting of a vertex $a$, a triangle $b_1, b_2, b_3$ (disjoint from $a$), and three paths $P_1$, $P_2$, $P_3$ from $a$ to $b_1, b_2, b_3$, respectively, 
such that any two of them induce a hole (so in particular, at least two of the three paths are of length at least 2).
A {\em prism} is a graph consisting of two disjoint triangles $a_1a_2a_3$ and $b_1b_2b_3$ and three paths $P_1$, $P_2, P_3$, with $P_i$ from $a_i$ to $b_i$, such that 
any two of them induce a hole.
Thetas, pyramids, and prisms are called {\em three-path configurations} (3PCs). 

We introduce the following notation, which we will use throughout the paper: 
\begin{itemize}
    \item $\mathcal{C}$ is the class of ($C_4$, diamond, theta, pyramid, prism, even wheel)-free graphs.

    \item $\mathcal{C}^*$ is the class of ($C_4$, diamond, theta, prism, even wheel)-free graphs (and in particular, $\mathcal{C} \subseteq \mathcal{C}^*$).
\end{itemize}

Our main result is the following.

\begin{theorem}\label{t1}
There exists an integer $\Gamma$ such that $\ta(G) \leq \Gamma$ for every graph $G \in \mathcal C$.
\end{theorem}

Theorem \ref{t1} strengthens the main result of \cite{tw4},
while building on some of the ideas of its proof:

\begin{theorem}\label{tw4-t} {\em (\cite{tw4})}
For each $t>0$, there exists an integer $c_t$ such that for every $G\in \mathcal{C}$ with $\omega(G) \leq t$, we have $\tw (G)\leq c_t$.
\end{theorem}

Since the class of (even hole, diamond, pyramid)-free graphs
is a subclass of $\mathcal{C}$ (as even-hole-free graphs are ($C_4$, theta, prism, even wheel)-free),
the above two theorems imply, respectively, the following two statements about (even hole, diamond, pyramid)-free graphs.

\begin{theorem}\label{t1cor}
There exists an integer $c$ such that, for every (even hole, diamond, pyramid)-free graph $G$, we have $\ta (G)\leq c$.
\end{theorem}

\begin{theorem}\label{tw4-tcor} {\em (\cite{tw4})}
For each $t>0$, there exists an integer $c_t$, such that for every (even hole, diamond, pyramid, $K_t$)-free graph $G$, $\tw (G)\leq c_t$.
\end{theorem}

Thus Theorem~\ref{t1cor} yields a 
polynomial-time algorithm for the MWIS problem in this class. It is in fact known that for a certain superclass of this class, the MWIS problem can be solved in polynomial time  \cite{pmcpaper},
which we discuss further below. But the algorithmic method shown here is of independent interest, since it might be possible to extend it to (even-hole, diamond)-free graphs (for which we do not yet have a polynomial time algorithm to solve the MWIS problem).

\begin{conjecture}\label{con:ehd}
The class of (even hole, diamond)-free graph has bounded tree-$\alpha$.
\end{conjecture}

\subsection*{Further context}

First, observe that the class $\mathcal C$ was already known to be $(\tw, \omega)$-bounded \cite{tw4}, so Theorem~\ref{t1} corroborates Conjecture~\ref{con:tw-omega-implies-ta} for this class. Although many of the ingredients of our proof are class-specific, it is worth investigating whether any of our methods can be used for tackling the conjecture in general.
 
The second consideration is an algorithmic one. Even-hole-free graphs have been studied extensively since the 1990s, and yet despite their structural similarity to perfect graphs, the complexity of some fundamental 
optimization problems that are solvable in polynomial time for perfect graphs is not known for even-hole-free graphs. Indeed, while we can find maximum cliques in even-hole-free graphs efficiently by a result shown by Farber \cite{farber} and later Alekseev \cite{alekseev} stating that $C_4$-free graphs have quadratically many maximal cliques, we still do not have efficient algorithms for the MWIS, minimum vertex coloring, and minimum clique cover problems. 

In this direction, a number of subclasses of even-hole-free graphs have been studied: for instance, in \cite{kmv} it was shown that (even hole, diamond)-free graphs can be colored in polynomial time. A sketch of that proof is as follows. First, a decomposition theorem is obtained for (even hole, diamond)-free graphs: these graphs are decomposed by star cutsets that partition into two cliques, and by 2-joins, into line graphs of a tree plus at most two vertices. This decomposition theorem is used to show that (even hole, diamond)-free graphs always have a vertex that is either simplicial or of degree 2. This then implies that the class is ``$\beta$-perfect,'' which in turn yields a polynomial time coloring algorithm which proceeds by coloring greedily on a particular, easily constructible
 ordering of the vertices. 

Our current paper continues this study by investigating the MWIS problem, in the hope of extending the methods presented here to (even hole, diamond)-free graphs (see Conjecture~\ref{con:ehd}).
We now survey some subclasses of even-hole-free graphs in relation to the MWIS
problem. Recall that even-hole-free graphs are ($C_4$, theta, prism, even wheel)-free.

\vspace{2ex}

\noindent
{\bf A superclass of (even hole, pyramid)-free graphs}

\vspace{2ex}

\noindent
In \cite{pmcpaper}, a superclass of (even hole, pyramid)-free graphs is studied, namely the class of ($C_4$, theta, prism, pyramid, turtle)-free graphs 
(where a {\em turtle} is a graph derived from a wheel that cannot be present in even-hole-free graphs, or in fact, in even-wheel-free graphs). This yields, in particular, a polynomial-time algorithm for MWIS in our class $\mathcal C$. However, their method (which we outline below) does not generalize beyond the pyramid-free setting. In contrast, we believe the method we present in this paper extends to the larger class of (even hole, diamond)-free graphs.

To understand why the results from \cite{pmcpaper} do not generalize, we outline their procedure for solving MWIS here. It is shown in \cite{pmcpaper} that the class under study has polynomially many minimal separators, and that the set of all minimal separators can be constructed in  polynomial time. 

The number of minimal separators of a graph is related to the number of potential maximal cliques (PMCs) introduced by  Bouchitt\'e and Todinca in
\cite{DBLP:journals/siamcomp/BouchitteT01, DBLP:journals/tcs/BouchitteT02}.
For a graph $G$, denote by $n$  the number of vertices, $m$ the number of edges, $s$ the number of minimal separators,  and $p$ the number of PMCs in  $G$.
In \cite{DBLP:journals/tcs/BouchitteT02}, it is proved that
$p$ is in $\mathcal{O} (ns^2+ns+1)$, and that given the list of minimal separators, the PMCs of $G$ can be listed in time $\mathcal{O}(n^2ms^2)$.
In \cite{DBLP:conf/soda/LokshantovVV14}, based on \cite{DBLP:conf/stacs/FominV10}, it is proved that given the list of PMCs, the MWIS problem can be solved in
$\mathcal{O} (n^5mp)$.
So, summing everything up, the ability to list all minimal separators in polynomial time implies a polynomial-time algorithm for the MWIS problem.

On the other hand, if pyramids are allowed in even-hole-free graphs, the polynomial minimal separator property is lost, as witnessed for example by a {\em $k$-pyramid}  (a graph
consisting of a clique $B=\{ b_1,\ldots, b_k\}$, a vertex $a$ and paths $P_i= a\ldots b_i$, for $i=1, \ldots ,k$, such that any 3 of them induce a pyramid). 

\vspace{2ex}

\noindent
{\bf A subclass of (even hole, pyramid)-free graphs}

\vspace{2ex}

\noindent
The class of (even hole, $K_3$)-free graphs is a proper subclass of (even hole, pyramid)-free graphs.
In \cite{cameron-vuskovic} it is shown that (even hole, $K_3$)-free graphs (and in fact (theta, even wheel, $K_3$)-free graphs)
have treewidth at most 5.
On the other hand, in \cite{layered-wheels}, a certain construction called a ``layered wheel'' is given, which shows that (even hole, $K_4$)-free graphs
do not have bounded treewidth. The construction is full of diamonds, so it was conjectured in \cite{layered-wheels} and proved in \cite{tw9}
that (even hole, diamond, $K_4$)-free graphs have bounded treewidth (in fact, the result from \cite{tw9} pertains to more general classes of even-hole-free graphs obtained by forbidding any clique and a certain type of graph which contains diamonds).

The question remains whether the following is true (it can be shown that  the outcome holds when forbidding, in addition, another family of graphs, called ``generalized $k$-pyramids'' \cite{tw10}).

\begin{conjecture}\label{con:ehKt}
The class of (even hole, $K_t$)-free graph has logarithmic treewidth.
\end{conjecture}

\subsection*{Terminology and notation}

If $G$ is a path or a cycle, the {\em length} of $G$ is $|E(G)|$. By $P_k$ we denote the path on $k$ vertices. If $P = p_1 \dd p_2 \dd \cdots \dd p_k$, then $P^* = P \setminus \{p_1, p_k\}$ denotes the {\em interior} of the path $P$, and $p_1, p_k$ are its {\em ends}.

Let $G$ be a graph and let $v \in V(G)$. The {\em open neighborhood of $v$ in $G$}, denoted $N_G(v)$, is the set of vertices of $V(G)$ adjacent to $v$. The {\em closed neighborhood of $v$ in $G$}, denoted $N_G[v]$, is the union of $\{v\}$ and $N_G(v)$. Let $X \subseteq V(G)$. The {\em open neighborhood of $X$ in $G$}, denoted $N_G(X)$, is the set of vertices of $V(G) \setminus X$ with a neighbor in $X$. The {\em closed neighborhood of $X$ in $G$}, denoted $N_G[X]$, is the union of $X$ and $N_G(X)$. When the graph $G$ is clear from context, we omit the subscript $G$ from the open and closed neighborhoods. If $X, Y \subseteq V(G)$, we say {\em $X$ is anticomplete to $Y$} if there are no edges with one endpoint in $X$ and one endpoint in $Y$. We say that $X$ {\em has a neighbor in $Y$} if $X$ is not anticomplete to $Y$. We say that $v$ is anticomplete to $X$ if $\{v\}$ is anticomplete to $X$. 

A {\em clique} is a set of pairwise adjacent vertices. 
The {\em clique number} of $G$, denoted by $\omega(G)$, is the size of a maximum clique in $G$. The {\em chromatic number} of $G$, denoted by $\chi(G)$, is the minimum number of sets of a partition of $V(G)$ into stable sets. The {\em clique cover number} of $G$, denoted by $\cchi(G)$, is the minimum number of sets of a partition of $V(G)$ into cliques (note that some sources use ``clique cover number'' to refer the minimum number of cliques required to cover the edges of $G$, and use ``node clique cover number'' or ``clique partition number'' for what we call ``clique cover number'' here; for the avoidance of doubt, all clique covers in this paper will be vertex clique covers, that is, partitions of a vertex set into cliques).

Let $(H, v)$ be a wheel. We call $v$ the {\em center}, or {\em hub} of the wheel. 
A {\em sector of $(H, v)$} is a path $P \subseteq H$ of length at least one such that the ends of $P$ are adjacent to $v$ and $P^*$ is anticomplete to $v$. A sector of $(H, v)$ is {\em long} if it has length greater than one, and short otherwise. 

If paths $P_1=a \dd \ldots \dd b_1$, $P_2=a \dd \ldots \dd b_2$ $P_3=a \dd \ldots \dd b_3$ induce a pyramid, then
we say that the vertex $a$ is  the {\em apex} of this pyramid. 

A {\em star cutset} of a graph $G$ is a set $C \subseteq V(G)$ such that $G \setminus C$ is not connected, and there exists $v \in V(G)$ with $v \in C \subseteq N[v]$. We call $v$ a {\em center} of the star cutset $C$.
We say that the star cutset is {\em full} if $C = N[v]$. A {\em clique cutset} is a star cutset that is also a clique.

\section{Tools}
\label{sec:tools}
In this section, we describe the tools and preliminary results needed to prove that graphs in $\mathcal{C}$ have bounded tree-$\alpha$.

\subsection{$\chi$-boundedness of the complements and the tree-clique cover number}

Many of the results from \cite{tw4} that we use provide bounds on the size of certain sets in terms of $\omega$, but in actuality the proofs work by exhibiting clique covers of those sets. Correspondingly, our arguments revolve around clique covers rather than stable sets. As such, it is natural to define the {\em tree-clique cover number} $\tc(G)$ of a graph $G$ analogously to $\ta$. More explicitly, we define the clique cover number of a tree decomposition $(T, \varphi)$ as the maximum value of $\cchi(G[\varphi(t)])$ over nodes $t \in V(T)$, and we let $\tc(G)$ be the minimum value of the clique cover number over all tree decompositions of $G$. 

We note that, for $C_4$-free graphs, statements about $\tc$ have $\ta$ analogues, and vice-versa. Indeed, in one direction, $\alpha(G) \leq  \cchi(G)$ for any graph $G$. In the other direction, by a result of Wagon \cite{wagon}, the class of $2K_2$-free graphs is $\chi$-bounded, meaning that there exists a non-decreasing function $f : \mathbb N^+ \to \mathbb N^+$ such that, for any $2K_2$-free graph $G$, $\chi(G) \leq f(\omega(G))$. In fact, from \cite{wagon}, one may take $f(\omega(G))$ to be ${\omega(G) + 1 \choose 2}$. We record this in a lemma, noting that the complement of a $C_4$-free graph $G$ is $2K_2$-free:

\begin{lemma}[\cite{wagon}] \label{lem:chi-bound}
    Let $G$ be a $C_4$-free graph. Then $\cchi(G) \leq {\alpha(G) + 1 \choose 2}$.
\end{lemma}

The above discussion immediately yields the following:

\begin{lemma} \label{lem:ta-vs-tc}
    Let $G$ be a $C_4$-free graph. Then $\ta(G) \leq \tc(G) \leq {\ta(G) + 1 \choose 2}$.
\end{lemma}

We remark that $C_4$-freeness (or at least, $\chi$-boundedness in the complement) is necessary to obtain a relationship between $\tc$ and $\ta$:

\begin{remark}
 For any $C > 0$, there exists a graph with $\ta \leq 2$ and $\tc \geq C$. 
\end{remark}
\begin{proof}
    It is pointed out in \cite{dms2} that, in any tree decomposition $(T, \varphi)$ of a graph $G$, there exists a vertex $v \in G$ and a node $t \in T$ with $N[v] \subseteq \varphi(t)$; in particular, if we start with a graph $G$ and construct $G^+$ from two copies of $G$ by adding all possible edges between them, for any tree decomposition of $G^+$, some bag will contain a copy of $G$. We can use this to construct an example as claimed. Start with any triangle-free graph $G$ with chromatic number greater than $C$ (such graphs are well-known to exist); take its complement $\overline{G}$, and construct the graph $\overline{G}^+$ as above. The resulting graph will have $\alpha(\overline{G}^+) \leq 2$ (since its complement is still triangle-free), but by the above discussion, some bag in any tree decomposition of $\overline{G}^+$ will have clique cover number more than $C$.
\end{proof}

Finally, we also have the following simple observation relating $\tc$ to treewidth:

\begin{remark}
\label{rem:tw-tc} 
For every graph $G$, $\tc(G)\leq \tw (G)+1$. 
\end{remark}

Throughout the paper, we will work with $\tc$ rather than $\ta$, with the understanding that our results can be stated in terms of $\ta$ via Lemma~\ref{lem:ta-vs-tc}.

\subsection{Treewidth, $\tc$, clique cutsets and (3PC, wheel)-free graphs}

 We next state the observations that clique cutsets do not affect treewidth, and $\tc$. We note that \cite{dms2} includes an analogous result for $\ta$. In view of the treewidth and $\ta$ results, the proof of the $\tc$ result is routine and left to the reader.
  
\begin{lemma}
\label{lemma:clique-cutsets-tw} 
{\em (\cite{clique-cutsets-tw})}
Let $G$ be a graph. Then $\tw(G)$ is equal to the maximum treewidth over all induced subgraphs of $G$ with no clique cutset. 
\end{lemma}

\begin{lemma}
\label{lemma:clique-cutsets-tc} 
Let $G$ be a graph. Then $\tc(G)$ is equal to the maximum $\tc$ over all induced subgraphs of $G$ with no clique cutset. 
\end{lemma}

In view of this, given a class $\mathcal{X}$, we will write $\mathcal{X}_{\text{nc}}$ for the class of graphs in $\mathcal X$ with no clique cutset.

\medskip

We continue with a few observations about (3PC, wheel)-free graphs.

\begin{lemma}
\label{lemma:us} 
{\em (\cite{cckv-us})}
If $G$ is (3PC, wheel)-free, then either $G$ has a clique cutset or $G$ is a complete graph or a hole.
\end{lemma}

\begin{lemma}
\label{lemma:us-tw} 
If $G$ is (3PC, wheel)-free, then $\tw (G)\leq \max\{ \omega(G)-1,2\}$ and  $\tc (G)\leq 2$.
In particular, if $G$ is (theta, wheel, $K_3$)-free, then  $tw(G)\leq 2$.
\end{lemma}
\begin{proof}
The first statement follows from Lemma \ref{lemma:clique-cutsets-tw},  Lemma \ref{lemma:clique-cutsets-tc} and Lemma \ref{lemma:us}, and the second from the fact
that (theta, wheel, $K_3$)-free graphs are a subclass of (3PC, wheel)-free graphs.
A different proof of the second statement is given in \cite{tw4}.
\end{proof}

\subsection{Balanced separators} 
Let $G$ be a graph. A {\em weight function on $G$} is a function $w:V(G) \to \mathbb{R}$. For $X \subseteq V(G)$, we let $w(X) = \sum_{x \in X} w(x)$. 
Let $G$ be a graph, let $w: V(G) \to [0, 1]$ be a weight function on $G$ with $w(G) = 1$, and let $c \in [\frac{1}{2}, 1)$. A set $X \subseteq V(G)$ is a  {\em $(w, c)$-balanced separator} if $w(D) \leq c$ for every component $D$ of $G \setminus X$. The next few lemmas state how $(w, c)$-balanced separators relate to $\tc$. They are modelled on analogous results concerning treewidth, which we also include.

\begin{lemma}[\cite{wallpaper, params-tied-to-tw}, see also \cite{reed}]\label{lemma:bs-to-tw}
Let $G$ be a graph, let $c \in [\frac{1}{2}, 1)$, and let $k$ be a positive integer. If, for every weight function $w: V(G) \to [0, 1]$ with $w(G) = 1$, the graph $G$ has a $(w, c)$-balanced separator of size at most $k$, then $\tw(G) \leq \frac{1}{1-c}k$. 
\end{lemma}

\begin{lemma}\label{lemma:bs-to-ta}
 Let $G$ be a $C_4$-free graph, let $c \in [\frac{1}{2}, 1)$, and let $k$ be a positive integer. Suppose that for every weight function $w: V(G) \to [0, 1]$ with $w(G) = 1$, $G$ has a $(w, c)$-balanced separator with clique cover number at most $k$. Then there exists a number $g(k, c)$ such that $\tc(G) \leq g(k, c)$. 
\end{lemma}

\begin{proof}
    We adapt the proof from \cite{reed}. 
    Put $h(k, c) := \lceil 2k/(1 - c)\rceil$, and $g(k, c) := {h(k, c) + 1 \choose 2} + k$. We claim that $g(k, c)$ satisfies the statement of the lemma. In fact, we will recursively show the following slightly stronger statement:
    
    \begin{itemize}
        \item[(*)] Let $(G, W)$ be a pair where $G$ is as in the lemma, and $W \subseteq V(G)$ has $\cchi(W) \leq {h(k, c) + 1 \choose 2}$. Then there exists a tree decomposition $(T, \varphi)$ of $G$ of clique cover number at most $g(k, c)$, and such that $W \subseteq \varphi(t)$ for some $t \in V(T)$. In particular, $\tc(G) \leq g(k, c)$. 
    \end{itemize}

    The statement (*) is clear if $\alpha(G) \leq h(k, c)$. Indeed, in that case, Lemma~\ref{lem:chi-bound} yields that a tree decomposition consisting of a single bag does the job.

    If $\alpha(G) > h(k, c)$, we may assume without loss of generality that $\alpha(W) \geq h(k, c)$ (otherwise, we simply add vertices to it arbitrarily until $\alpha(W) = h(k, c)$, and by Lemma~\ref{lem:chi-bound}, we do not lose the property that $\cchi(W) \leq {h(k, c) + 1 \choose 2}$ in doing so). We then select a stable subset $W'$ of $W$ with $|W'| = h(k, c)$, and define a weight function $w : V(G) \to [0, 1]$ by putting $w(x) := \frac{1}{h(k, c)}$ for each $x \in W'$, and $w(x) := 0$ for all other $x$. By assumption, $G$ has a $(w, c)$-balanced separator $X$ with $\cchi(X) \leq k$.

    Now for any component $D$ of $G \setminus X$, we have $w(D) \leq c$, and so there are at least $(1-c)h(k, c) \geq 2k$ vertices of $W'$ outside of $D$. In particular, since $\alpha(X) \leq \overline{\chi}(X) \leq k$, there are $k$ vertices of $W'$, say $x_1, \dots, x_k$, which do not belong to $D \cup X$. 

    Let us now pick a clique cover of $W$ with at most ${h(k, c) + 1 \choose 2}$ cliques, and consider the set $Y := W \cap D$. By the above, the (distinct) cliques containing $x_1, \dots, x_k$ do not intersect $Y$, so $\cchi(Y) \leq {h(k, c) + 1 \choose 2} - k$. This shows $\cchi(Y \cup X) \leq \cchi(Y) + \cchi(X) \leq {h(k, c) + 1 \choose 2}$, and we may recursively find a tree decomposition of $D \cup X$ by applying (*) to the pair $(D \cup X, Y \cup X)$ (noting that $D \cup X$ is a smaller graph).

    We do the above for each component $D_i$ of $G \setminus X$, defining $Y_i$ analogously to $Y$; we obtain, for each $i$, a tree decomposition $(T_i, \varphi_i)$ of $D_i \cup X$ with a node $t_i \in T_i$ such that $Y_i \cup X \subseteq \varphi_i(t_i)$. We then assemble those tree decompositions into a decomposition $(T, \varphi)$ of $G$ with the desired property by starting with the union of the $T_i$'s and adding single vertex $t$ adjacent to all $t_i$'s, and with $\varphi(t) = W \cup X$ (and with $\varphi$ restricting to each $\varphi_i$ on the corresponding $T_i$). Note that $\cchi(G[\varphi(t)]) = \cchi(W \cup X) \leq \cchi(W) + \cchi(X) \leq g(k, c)$. Inductively, this is also true of every other bag. This yields a tree decomposition of $G$ (see Fact~2.8 from page 111 of \cite{reed}).
    
\end{proof}

In the converse direction, we have the following:

\begin{lemma}[\cite{cygan}]
\label{lemma:tw-to-weighted-separator}
Let $G$ be a graph and let $k$ be a positive integer. If $\tw(G) \leq k$, then for every weight function $w:V(G) \to [0, 1]$ with $w(G) = 1$ and every $c \in [\frac{1}{2}, 1)$, $G$ has a $(w, c)$-balanced separator of size at most $k + 1$.
\end{lemma}

This can be generalized at no extra cost to obtain $\ta$ and $\tc$ versions of the lemma (since the separator in the proof from \cite{cygan} is contained in a single bag of the given tree decomposition). 

\begin{lemma}[see \cite{cygan}]
\label{lem:ta-to-weighted-separator}
 Let $G$ be a graph and let $k$ be a positive integer. If $\ta(G) \leq k$, then for every weight function $w:V(G) \to [0, 1]$ with $w(G) = 1$ and every $c \in [\frac{1}{2}, 1)$, $G$ has a $(w, c)$-balanced separator of independence number at most $k$. The analogous statement for $\tc$ holds as well.
\end{lemma}

\subsection{Separations} 
A {\em separation} of a graph $G$ is a triple $(A, C, B)$, where $A, B, C \subseteq V(G)$, $A \cup C \cup B = V(G)$, $A$, $B$, and $C$ are pairwise disjoint, and $A$ is anticomplete to $B$. If $S = (A, C, B)$ is a separation, we let $A(S) = A$, $B(S) = B$, and $C(S) = C$. Two separations $(A_1, C_1, B_1)$ and $(A_2, C_2, B_2)$ are {\em nearly non-crossing} if every component of $A_1 \cup A_2$ is a component of $A_1$ or a component of $A_2$. A separation $(A, C, B)$ is a {\em star separation} if there exists $v \in C$ such that $C \subseteq N[v]$. Let $S_1 = (A_1, C_1, B_1)$ and $S_2 = (A_2, C_2, B_2)$ be separations of $G$. We say $S_1$ is a {\em shield for $S_2$} if $B_1 \cup C_1 \subseteq B_2 \cup C_2$. 

\medskip

We note the following result from \cite{tw4}:

\begin{lemma}[\cite{tw4}]
\label{lemma:shields}
Let G be a ($C_4$, diamond)-free graph with no clique cutset, let $v_1, v_2 \in V(G)$, and let $S_1 = (A_1, C_1, B_1)$ and $S_2 = (A_2, C_2, B_2)$ be star separations of $G$ such that $v_i \in C_i \subseteq N[v_i]$, $B_i$ is connected, and $N(B_i) = C_i \setminus \{v_i\}$ for $i = 1, 2$. Suppose that $v_2 \in A_1$ and $B_2 \cap (B_1 \cup (C_1 \setminus \{v_1\})) \neq \emptyset$. Then, $S_1$ is a shield for $S_2$. 
\end{lemma}

Let $G$ be a graph and let $w: V(G) \to [0, 1]$ be a weight function on $G$ with $w(G) = 1$. A vertex $v \in V(G)$ is called {\em balanced} if $w(D) \leq \frac{1}{2}$ for every component $D$ of $G \setminus N[v]$, and {\em unbalanced} otherwise. Let $U$ denote the set of unbalanced vertices of $G$. Let $v \in U$. The {\em canonical star separation for $v$}, denoted $S_v = (A_v, C_v, B_v)$, is defined as follows: $B_v$ is the connected component of $G \setminus N[v]$ with largest weight, $C_v = \{v\} \cup \left(N(v) \cap N(B_v)\right)$, and $A_v = V(G) \setminus (B_v \cup C_v)$. Note that $B_v$ is well-defined since $v \in U$. 

Let $\leq_A$ be the relation on $U$ where for $x, y \in U$, $x \leq_A y$ if and only if $x = y$ or $y \in A_x$. 
We also need:

\begin{lemma}[\cite{tw4}]
\label{lemma:leqA-partial-order}
Let $G$ be a ($C_4$, diamond)-free graph with no clique cutset, let $w:V(G) \to [0, 1]$ be a weight function on $G$ with $w(G) = 1$, let $U$ be the set of unbalanced vertices of $G$, and let $\leq_A$ be the relation on $U$ defined above. Then, $\leq_A$ is a partial order. 
\end{lemma}

\subsection{Central bags} 
Let $G$ be a graph and let $w:V(G) \to [0, 1]$ be a weight function on $G$ with $w(G) = 1$. We call a collection $\S$ of separations of $G$ {\em smooth} if the following hold: 
\begin{enumerate}[(i)]
    \item \label{smooth:non-crossing} $S_1$ and $S_2$ are nearly non-crossing for all distinct $S_1, S_2 \in \S$; 
    \item There is a set of unbalanced vertices $v(\S) \subseteq V(G)$ such that there is a bijection $f$ from $v(\S)$ to $\S$ with $v \in C(f(v)) \subseteq N[v]$ and $A(f(v)) \subseteq A_v$ for each $v \in v(\S)$;
    \item \label{smooth:v-cap-A-empty} $v(\S) \cap A(S) = \emptyset$ for all $S \in \S$. 
\end{enumerate} 

Let $\S$ be a smooth collection of separations of $G$. Then, the {\em central bag for $\S$}, denoted $\beta_\S$, is defined as follows: 
$$\beta_\S = \bigcap_{S \in \S} (B(S) \cup C(S)).$$
By property \eqref{smooth:v-cap-A-empty} of smooth collections of separations, it holds that $v(\S) \subseteq \beta_\S$. Also note that $G \setminus \beta_\S = \bigcup_{S \in \S} A(S)$. Moreover, by property \eqref{smooth:non-crossing}, every component $D$ of $G \setminus \beta_{\S}$ is contained in $A(f(v_i))$ for some $i$ (to see this, we consider a maximal subset of $D$ belonging to a single $A(f(v_i))$, and use \eqref{smooth:non-crossing} to extend the subset if it does not equal the whole of $D$, contradicting maximality). 

\medskip

We next define an {\em inherited weight function} $w_\S$ on $\beta_\S$ by recording the weight of each component of $G \setminus \beta_\S$ into a unique vertex $v_i \in \beta_\S$. To this end, we start by fixing an ordering $\{v_1, \hdots, v_k\}$ of $v(\S)$. For every $f(v_i) \in \S$, let $A^*(f(v_i))$ be the union of all connected components $D$ of $\bigcup_{1 \leq j \leq k} A(f(v_j))$ such that $D \subseteq A(f(v_i))$, and $D \not \subseteq A(f(v_j))$ for every $j < i$. In particular, $(A^*(f(v_1)), \hdots, A^*(f(v_k)))$ is a partition of $\bigcup_{S \in \S} A(S)$. For a component $D$ of $\bigcup_{S \in \S} A(S)$, we call the unique $v_i$ with $D \subseteq A^*(f(v_i))$ the {\em anchor} of $D$. In other words, the anchor of $D$ is the first $v_i$ in our ordering with $D \subseteq A(f(v_i))$. We define $w_\S : V(\beta_\S) \to [0, 1]$ by $w_\S(v_i) = w(v_i) + w(A^*(f(v_i)))$ for all $v_i \in v(\S)$, and $w_\S(v) = w(v)$ for all $v \not \in v(\S)$.

 \subsection{Wheels}

We will need the following result from \cite{tw4}.

\begin{lemma}[\cite{tw4}]
\label{lemma:common_nbrs}
Let $G$ be a (theta, even wheel)-free graph, let $H$ be a hole of $G$, and let $v_1, v_2 \in V(G) \setminus V(H)$ be adjacent vertices each with at least two non-adjacent neighbors in $H$. Then, $v_1$ and $v_2$ have a common neighbor in $H$. 
\end{lemma}

Recall that a {\em wheel} $(H, w)$ is a hole $H$ and a vertex $w \in V(G)$ such that $w$ has at least three neighbors in $H$. A wheel $(H, w)$ is a {\em twin wheel} if $N(w) \cap H$ is a path of length two. A wheel $(H, w)$ is a {\em bug} if $|N(w) \cap H| = 3$ and $w$ has exactly two adjacent neighbors in $H$. Bugs are also known as {\em short pyramids}. 
 A wheel $(H, w)$ is a {\em universal wheel} if $w$ is complete to $H$. We will also use the following result about wheels and star cutsets. 

\begin{lemma}[\cite{tw4}]
\label{lemma:no_wheels}
Let $G$ be a ($C_4$, even wheel, theta, prism)-free graph and let $(H, v)$ be a wheel of $G$ that is not a bug, a twin wheel nor a universal wheel. Let $(A, C, B)$ be a separation of $G$ such that $v \in C \subseteq N[v]$, $B$ is connected, and $N(B) = C \setminus \{v\}$. Then, $H \not \subseteq B \cup C$. 
\end{lemma}

\subsection{Trisimplicial elimination orderings} \label{sec:bisimplicial}

A vertex $v$ of a graph $G$ is called {\em simplicial} if its neighborhood is a clique, {\em bisimplicial} if there exist cliques $K_1, K_2$ of $G$ such that $N(v) = K_1 \cup K_2$, and {\em trisimplicial} if there exist cliques $K_1, K_2, K_3$ of $G$ such that $N(v) = K_1 \cup K_2 \cup K_3$. 

A central ingredient in our argument is the fact that the graphs we consider admit elimination orderings $v_1, \dots, v_{|V(G)|}$ of their vertex sets such that for each $i$, $N(v_i) \cap \{v_i, v_{i + 1}, \dots, v_{|V(G)|}\}$ is the union of a bounded number of cliques (in particular, bisimplicial or trisimplicial elimination orderings qualify).

It is worth pointing out that every even-hole-free graph has a bisimplicial vertex \cite{ehf-bisimplicial}, and as such admits a bisimplicial elimination ordering. This gives us hope that it might be possible to generalize the results of this paper in order to obtain a polynomial algorithm for MWIS in the whole class of even-hole-free graphs. However, for the purposes of this paper, we would like to derive and use an analogous result for the class $\mathcal C$. It is not clear how to adapt the proof from \cite{ehf-bisimplicial} to this setting; nonetheless, there is another result from \cite{kmv} that one may use as a basis for what we need here: 

\begin{theorem}\label{os2}
{\em (\cite{kmv})}
Every (even hole, diamond)-free graph has a vertex which is either simplicial or of degree 2.
\end{theorem}

We note that this result is in a sense orthogonal to the one from \cite{ehf-bisimplicial}: it requires the additional condition of diamond-freeness, and yields a stronger outcome than just a bisimplicial vertex. While the proof of this result does generalize to the class $\mathcal C$ (and in fact to the even more general class of ($C_4$, diamond, theta, prism, even wheel)-free graphs) in a relatively straightforward way, it is a bit long. Instead, we include here a shorter proof of a weaker result which is still strong enough for our purposes. 

\begin{lemma}\label{lem:trisimplicial}
    Every graph in $\mathcal C$ has a trisimplicial vertex.
\end{lemma}
\begin{proof}
    We first prove that graphs in $\mathcal C$ with no wheel satisfy a stronger statement:

    \sta{\label{nowheel} If a graph $G \in \mathcal C$ contains no wheel, then it has a bisimplicial vertex.}

    Note that $G$ is (3PC, wheel)-free, so by Lemma~\ref{lemma:us}, it either has a clique cutset, or is a complete graph or a hole. In the latter two cases, $G$ satisfies the statement, so assume $G$ has a clique cutset. In that case, it is well-known (see e.g.\ Lemma~3.3 of \cite{vuskovic-truemper}) that $G$ has an extreme clique cutset $C$, that is, a clique cutset $C$ such that, for some component $D$ of $G \setminus C$, $D \cup C$ has no clique cutset. As above, $D \cup C$ is a complete graph or a hole; in particular, any vertex in $D$ is bisimplicial in $G$, proving \eqref{nowheel}.
        
    \bigskip

    We are ready to show that general graphs in $\mathcal C$ have trisimplicial vertices. We do so using an approach conceptually similar to the proof of \eqref{nowheel}, by finding an ``extreme'' star cutset in $G$.

    \bigskip
    
    If $G$ is wheel-free, then we are done by \eqref{nowheel}, so assume $G$ contains a wheel $(H, w)$.  Since $G$ is pyramid-free, $(H, w)$ is not a bug, and since it is diamond-free, it is not a twin or universal wheel, and hence Lemma~\ref{lemma:no_wheels} implies that $G$ has a star cutset. Let us pick a star cutset $C$ with center $v$ such that some component $D$ of $G \setminus C$ is minimal, in the sense that it does not properly contain a component of $G \setminus C'$ for any other choice of star cutset $C'$.

    \begin{itemize}
        \item[] \textit{Case 1:} $v$ is complete to $D$. In this case, $D$ is a singleton (since otherwise we may add all but one vertex of $D$ to the cutset, and removing the new cutset leaves a singleton component properly contained in $D$, contradicting the minimality of $D$). But in this case, the unique vertex in $D$ is simplicial in $G$ by diamond-freeness, since its neighborhood is contained in $N[v]$, and we are done.
        \medskip
        \item[] \textit{Case 2:} $v$ is not complete to $D$. In this case, once more by the minimality of $D$, $v$ must be anticomplete to $D$ (since otherwise we may add all neighbors of $v$ in $D$ to the cutset and obtain once more a smaller component). We now claim:

        \sta{\label{Dnowheel} $G[D]$ is wheel-free.} 
     
        Indeed, suppose not, and let $(H, u)$ be a wheel in $G[D]$. Put $C' := N[u]$, and note that, by Lemma~\ref{lemma:no_wheels}, $C'$ is a star cutset in $G$ with center $u$, and 
        $H \nsubseteq Q \cup C'$ for any component $Q$ of $G \setminus C'$ (in actuality, the statement of the lemma yields $H \nsubseteq Q \cup N(Q) \cup \{u\}$, but by diamond-freeness, we cannot have $H \setminus (Q \cup N(Q) \cup \{u\}) \subseteq C' \setminus N(Q)$, since if that were the case, any vertex in $H \setminus (Q \cup N(Q) \cup \{u\})$ would be the center of a $P_3$ in $H \cap N(u)$). Let $B'$ be the component of $G \setminus C'$ containing $v$. By the above, there exists $x \in D \setminus (B' \cup C')$. Let $D'$ be the component of $G \setminus C'$ containing $x$. Since every path from $x$ to $G \setminus D$ intersects $C$, and since $C \subseteq B' \cup C'$, it follows that $D' \subseteq D$, and this containment is proper (as $u \in D \setminus D'$), contradicting the choice of $C$ and $D$.   

        \bigskip
        
        Now by \eqref{nowheel}, $G[D]$ has a bisimplicial vertex $x$. But $x$ has at most one neighbor in $G \setminus D$, since $N_{G \setminus D}(x) \subseteq N(v)$, and $x$ and $v$ are non-adjacent (and so by ($C_4$, diamond)-freeness, they have at most one common neighbor). Thus $x$ is trisimplicial in $G$, completing the proof.
    \end{itemize}

\end{proof}

 \section{Balanced separators and central bags}
 \label{sec:bs-in-cbags}
In this section, we construct a useful central bag for graphs in $\mathcal{C}$ and prove that the central bag has a balanced separator of bounded $\cchi$. We note that, because of Lemma \ref{lemma:clique-cutsets-tc}, we often assume that the graphs we work with do not have clique cutsets. We also have the following simple characterisation of the neighborhood of vertices in diamond-free graphs. 
\begin{lemma}[\cite{tw4}]
\label{lemma:clique-nbrs}
Let $G$ be diamond-free and let $v \in V(G)$. Then, $N(v)$ is the union of disjoint and pairwise anticomplete cliques. 
\end{lemma}

For $X \subseteq V(G)$, let $\Hub(X)$ denote the set of all vertices $x \in X$ for which there exists a wheel $(H, x)$ which is not a bug and with $H \subseteq X$. We will need several results from \cite{tw4} that we will now describe, introducing the relevant notation and terminology as necessary. Note that we will not explicitly use Lemmas~\ref{lemma:bounding-nbrhood-helper}~and~\ref{lemma:new-sepns-noncrossing} in this paper, but for completeness (and to give the reader some intuition of how the results that we do use explicitly were obtained), we include their statements. We start with a result showing that, under suitable conditions, the components of a central bag by a star cutset attach to the cutset in a controlled way:

\begin{lemma}[\cite{tw4}]
\label{lemma:bounding-nbrhood-helper}
Let $G$ be a ($C_4$, theta, prism, even wheel, diamond)-free graph with no clique cutset and let $w:V(G) \to [0, 1]$ be a weight function on $G$ with $w(G) = 1$. Let $\S$ be a smooth collection of separations of $G$, let $\beta_\S$ be the central bag for $\S$, and let $w_\S$ be the inherited weight function on $\beta_\S$. Let $v \in \beta_\S$ and (by Lemma \ref{lemma:clique-nbrs}) let $N_{\beta_\S}(v) \setminus \Hub(\beta_\S) =K_1 \cup \cdots \cup K_t$, where $K_1, \hdots, K_t$ are disjoint, pairwise anticomplete cliques. Assume that $v$ is not a pyramid apex in $\beta_\S$. Let $D$ be a component of $\beta_\S \setminus N[v]$. Then, at most two of $K_1, \hdots, K_t$ have a neighbor in $D$. 
\end{lemma}

We next have a result about balanced separators in central bags with balanced vertices. This result follows directly from the proof of Lemma~20 of \cite{tw4}. The statement of that lemma gives a separator of size $6\omega(\beta_s) + k$, under the stronger assumption that $|N(v) \cap \Hub(\beta_{\S})| < k$, but its proof works by actually providing a clique cover of $N(v) \setminus \Hub(\beta_{\S})$ by 6 cliques, which is strong enough to derive the following: 
\begin{lemma}[\cite{tw4}]
\label{lemma:balanced_vtx_bs}
Let $G$ be a ($C_4$, theta, pyramid, prism, diamond, even wheel)-free graph with no clique cutset, let $w:V(G) \to [0, 1]$ be a weight function on $G$ with $w(G) = 1$, let $c \in [\frac{1}{2}, 1)$, and let $k$ be a positive integer.
Let $\S$ be a smooth collection of separations of $G$, let $\beta_\S$ be the central bag for $\S$, and let $w_\S$ be the inherited weight function on $\beta_\S$. Suppose that there exists $v \in \beta_\S$ such that $v$ is balanced in $G$, and assume that $\cchi(N(v) \cap \Hub(\beta_\S)) \leq k$. Then $\beta_\S$ has a $(w_\S, c)$-balanced
separator of clique cover number at most $6 + k$.
\end{lemma}

Recall that $\CC^*$ denotes the class of ($C_4$, diamond, theta, prism, even wheel)-free graphs with no clique cutset. Let $G \in \CC^*$, let $w:V(G) \to [0, 1]$ be a weight function on $G$ with $w(G) = 1$, and let $U$ be the set of unbalanced vertices of $G$. Let $X \subseteq U$. The {\em $X$-revised collection of separations}, denoted $\tilde{\S}_X$, is defined as follows. Let $u \in X$, and let $\tilde{S}_u = (\tilde{A}_u, \tilde{C}_u, \tilde{B}_u)$ be such that $\tilde{B}_u$ is the largest-weight connected component of $G \setminus N[u]$, $\tilde{C}_u = C_u \cup \bigcup_{v \in X \cap (C_u \setminus \{u\})} (N(u) \cap N(v))$, and $\tilde{A}_u = V(G) \setminus (\tilde{C}_u \cup \tilde{B}_u)$. Then, $\tilde{\S}_X = \{\tilde{S}_u : u \in X\}$. Note that the separations in $\tilde{\S}_X$ are closely related to canonical star separations. Specifically, for all $u \in X$, the following hold: 
\begin{enumerate}[(i)]
    \item $\tilde{B}_u = B_u$, 
    \item $C_u \subseteq \tilde{C}_u \subseteq N[u]$, 
    \item $\tilde{A}_u \subseteq A_u$, 
    \item $A_u \setminus N(u) \subseteq \tilde{A}_u$.
\end{enumerate}

One way of thinking about the revised separations is that we begin with the canonical separation, and extend $C_u$ by extending every edge $ux$ with $x \in C_u \cap X$ to a maximal clique (each such clique has size at least 2, so it can be extended to a maximal clique in a unique way, by diamond-freeness). 

\bigskip

    The following lemma shows that separations from a revised collection are nearly non-crossing. It follows directly from Lemma~21 of \cite{tw4}, which additionally assumes $K_t$-freeness in the statement, but never uses that assumption in the proof.

 \begin{lemma}[\cite{tw4}] \label{lemma:new-sepns-noncrossing}
 Let $G \in \CC^*$, let $w: V(G) \to [0, 1]$ be a weight function on $G$ with $w(G) = 1$, let $U$ be the set of unbalanced vertices of $G$, and let $X \subseteq U$ be such that every vertex of $X$ is minimal under the relation $\leq_A$. Let $\tilde{\S} = \tilde{\S}_X$ be the $X$-revised collection of separations. Then, $\tilde{S}_u$ and $\tilde{S}_v$ are nearly non-crossing for all $\tilde{S}_u, \tilde{S}_v \in \tilde{\S}$. 
 \end{lemma}

We now show how to construct a useful collection of separations of $G$. 
 By Theorem~\ref{lem:trisimplicial}, any graph in $\mathcal C$ admits a trisimplicial elimination ordering. Let $G \in \mathcal{C}$, and let $v_1, \dots, v_\ell$ be a trisimplicial elimination ordering of $\Hub(G)$. Let $U$ be the set of unbalanced vertices of $G$. Let $m$ be defined as follows. If $\Hub(G) \subseteq U$, then $m=\ell+1$. Otherwise, let $m$ be minimum such that $v_m$ is an element of $\Hub(G) \setminus U$. Now, $\{v_1, \hdots, v_{m-1}\} \subseteq U$. Let $M$ be the set of minimal vertices of $\{v_1, \hdots, v_{m-1}\}$ under the relation $\leq_A$, and let $\tilde{S}_M$ be the $M$-revised collection of separations. We call $(\{v_1, \hdots, v_\ell\}, m, M, \tilde{S}_M)$ the {\em hub division} of $G$. The next two lemmas describe properties of the hub division. Once more, their statements in \cite{tw4} assume $K_t$-freeness, but their proofs do not. Moreover, the hub division was defined in \cite{tw4} with respect to a different ordering of the vertices, but the proofs of these two lemmas apply to any ordering (and in particular, to our trisimplicial elimination ordering).

\begin{lemma}[\cite{tw4}]
\label{lemma:S_M-smooth}
Let $G \in \mathcal{C}$, let $w:V(G) \to [0, 1]$ be a weight function on $G$ with $w(G) = 1$, and let $(\{v_1, \hdots, v_\ell\}, m, M, \tilde{S}_M)$ be the hub division of $G$. Then, $\tilde{S}_M$ is a smooth collection of separations of $G$. 
\end{lemma}

By Lemma \ref{lemma:S_M-smooth}, there is a central bag $\beta_M$ for $\tilde{S}_M$ and an inherited weight function $w_M$ on $\beta_M$. 

\begin{lemma}[\cite{tw4}]
Let $G \in \mathcal{C}$, let $w:V(G) \to [0, 1]$ be a weight function on $G$ with $w(G) = 1$, and let $(\{v_1, \hdots, v_\ell\}, m, M, \tilde{S}_M)$ be the hub division of $G$. Let $\beta_M$ be the central bag for $\tilde{S}_M$ and let $w_M$ be the inherited weight function on $\beta_M$. Then, for all $1 \leq i \leq m - 1$, $v_i \notin \Hub(\beta_M)$. 
\label{lemma:no-wheels-in-beta}
\end{lemma}

We continue with a short result compiling two earlier ones.

\begin{lemma}
\label{lem:wheel-free}
Let $G$ be a (3PC, wheel)-free graph and let  $w:V(G) \to [0, 1]$ be a weight function on $G$. Then $G$ has a $(w, \frac{1}{2})$-balanced separator of clique cover number at most $2$. 
\end{lemma}
\begin{proof}
By Lemma~\ref{lemma:us-tw}, $\tc(G) \leq 2$. By Lemma \ref{lem:ta-to-weighted-separator}, $G$ has a $(w, \frac{1}{2})$-balanced separator of clique cover number at most 2. 
\end{proof}
Finally, we prove the main result of this section: that if $G$ is pyramid-free, then $\beta_M$ has a balanced separator of small clique cover number. 
\begin{theorem}
\label{thm:mainthm-betabs}
Let $G \in \CC$, let $w:V(G) \to [0, 1]$ be a weight function on $G$, and let $(\{v_1, \hdots, v_\ell\},$ $m, M, \tilde{S}_M)$ be the hub division of $G$. Let $\beta_M$ be the central bag for $\tilde{S}_M$ and let $w_M$ be the inherited weight function on $\beta_M$. Then, $\beta_M$ has a $(w_M, \frac{1}{2})$-balanced separator of clique cover number at most $9$.
\end{theorem}
\begin{proof}
First, suppose that $m = \ell+1$. Then, by Lemma \ref{lemma:no-wheels-in-beta}, for every $v \in \Hub(G)$, $v \notin \Hub(\beta_M)$. Since $\Hub(\beta_M) \subseteq \Hub(G)$, and since $G$ is pyramid- (and therefore bug-)free, it follows that $\beta_M$ is wheel-free. By Lemma \ref{lem:wheel-free}, $\beta_M$ has a $(w_M, \frac{1}{2})$-balanced separator of clique cover number at most 2. 

Now, assume $m < \ell + 1$. We claim that $v_m \in \beta_M$. Suppose that $v_m \in A_{v_i}$ for some $v_i \in M$. Then, $N[v_m] \subseteq A_{v_i} \cup C_{v_i}$, so $B_{v_i}$ is contained in a connected component $D$ of $G \setminus N[v_m]$. Since $v_i \in M$, it follows that $v_i$ is unbalanced, so $w(B_{v_i}) > \frac{1}{2}$. But now $w(D) > \frac{1}{2}$, so $v_m$ is unbalanced, a contradiction. Therefore, $v_m \not \in A_{v_i}$ for all $v_i \in M$. Since for all $v_i \in M$ it holds that $\tilde{A}_{v_i} \subseteq A_{v_i}$, it follows that $v_m \in \tilde{B}_{v_i} \cup \tilde{C}_{v_i}$, and so $v_m \in \beta_M$.

Next, consider $N(v_m) \cap \Hub(\beta_M)$. By Lemma \ref{lemma:no-wheels-in-beta},  $\Hub(\beta_M) \subseteq \{v_{m}, v_{m+1}, \hdots, v_\ell\}$. Therefore, $\cchi(N(v_m) \cap \Hub(\beta_M)) \leq 3$. Finally, by Lemma \ref{lemma:S_M-smooth}, $\tilde{S}_M$ is a smooth collection of separations of $G$. Now, by Lemma \ref{lemma:balanced_vtx_bs}, $\beta_M$ has a $(w_M, \frac{1}{2})$-balanced separator of clique cover number at most $6 + 3 = 9$.  
\end{proof}

\section{Extending balanced separators}
\label{sec:extending-bs}
In this section, we prove that we can construct a balanced separator of $G$ of bounded $\cchi$ given a balanced separator of $\beta_M$ of bounded $\cchi$. This is where our proof differs significantly from \cite{tw4}. Together with the main result of the previous section, this is sufficient to prove Theorem~\ref{t1}. First, we need the following lemma, which is once more an adaptation of a corresponding result from \cite{tw4}, and is directly implied by the proof from there. 

\begin{lemma} \cite{tw4}
Let $G \in \CC$, let $w:V(G) \to [0, 1]$ be a weight function on $G$ with $w(G) = 1$, and let $(\{v_1, \hdots, v_\ell\}, m, M, \tilde{S}_M)$ be the hub division of $G$. Let $\beta_M$ be the central bag for $\tilde{S}_M$. Let $v \in M$ be such that $v$ is not a pyramid apex in $\beta_M$. Then $\cchi(N_{\beta_M}(v) \setminus \Hub(\beta_M))) \leq 2$. 
\label{lemma:small-nbrs}
\end{lemma}

We also need the next lemma, which examines how three vertices can have neighbors in a connected subgraph. 
 \begin{lemma}[\cite{wallpaper}]
\label{lemma:three_vtx_attachments} 
Let $x_1, x_2, x_3$ be three distinct vertices of a graph $G$. Assume that $H$ is a connected
induced subgraph of $G \setminus \{x_1, x_2, x_3\}$ such that $H$ contains at least one neighbor of each of $x_1, x_2,
x_3$, and that subject to these conditions $V(H)$ is minimal subject to inclusion. Then one of the
following holds:
\begin{enumerate}[(i)]
    \item For distinct $i, j, k \in \{1, 2, 3\}$, there exists $P$ that is either a path from $x_i$ to $x_j$ or a
hole containing the edge $x_ix_j$ such that
\begin{itemize}
\item $H = P \setminus \{x_i, x_j\}$, and
\item either $x_k$ has at least two non-adjacent neighbors in $H$ or $x_k$ has exactly two neighbors
in H and its neighbors in H are adjacent.
\end{itemize}
\item There exists a vertex $a \in H$ and three paths $P_1, P_2, P_3$, where $P_i$ is from $a$ to $x_i$, such that
\begin{itemize}
\item $H = (P_1 \cup P_2 \cup P_3) \setminus \{x_1, x_2, x_3\}$, and
\item the sets $P_1 \setminus \{a\}$, $P_2 \setminus \{a\}$ and $P_3 \setminus \{a\}$ are pairwise disjoint, and
\item for distinct $i, j \in \{1, 2, 3\}$, there are no edges between $P_i \setminus \{a\}$ and $P_j \setminus \{a\}$, except
possibly $x_ix_j$.
\end{itemize}
\item There exists a triangle $a_1a_2a_3$ in $H$ and three paths $P_1, P_2, P_3$, where $P_i$ is from $a_i$ to $x_i$,
such that
\begin{itemize}
\item $H = (P_1 \cup P_2 \cup P_3) \setminus \{x_1, x_2, x_3\}$, and
\item the sets $P_1$, $P_2$, and $P_3$ are pairwise disjoint, and
\item for distinct $i, j \in \{1, 2, 3\}$, there are no edges between $P_i$ and $P_j$, except $a_ia_j$ and
possibly $x_ix_j$. 
\end{itemize}
\end{enumerate}
\end{lemma}

Now we prove the main result of this section. 
\begin{theorem}
For any $t \in \mathbb N^+$, there exists $\Gamma'(t)$ with the following property. Let $G \in \CC$, let $w:V(G) \to [0, 1]$ be a weight function on $G$ with $w(G) = 1$, and let $(\{v_1, \hdots, v_\ell\},$ $m, M, \tilde{S}_M)$ be the hub division of $G$. Let $\beta_M$ be the central bag for $\tilde{S}_M$ and let $w_M$ be the inherited weight function on $\beta_M$. Assume that $\beta_M$ has a $(w_M, \frac{1}{2})$-balanced separator of clique cover number at most $t$. 
Then $G$ has a $(w, \frac{1}{2})$-balanced separator whose clique cover number is at most $\Gamma'(t)$. 
\label{thm:extending-bs}
\end{theorem}

\setcounter{tbox}{0}
\begin{proof}

Let $X$ be a  $(w_M, \frac{1}{2})$-balanced separator of $\beta_M$ with $\cchi(X) \leq t$. Let $K_1, \dots, K_t$ be cliques covering $X$. Now for each $i$ with $1 \leq i \leq t$, let $\tilde{K_i}$ equal $K_i$ if $|K_i| = 1$, and let $\tilde{K_i}$ be a maximal clique in $G$ containing $K_i$ otherwise. We note that, since $G$ is diamond-free, $\tilde{K_i}$ is unique. Put $\tilde X := \bigcup\limits_{i = 1}^t \tilde{K_i}$, and note that $\tilde{X} \cap \beta_M$ is also a $(w_M, \frac{1}{2})$-balanced separator of $\beta_M$, with $\cchi(\tilde{X} \cap \beta_M) \leq t$. 
We claim that we have: 

\sta{\label{alpha-4} For each $x \in \{v_1, \dots, v_{m - 1}\}$, $\cchi(N_{\beta_M}(x)) \leq 5$.} 

We first show that this is true if $x \in M$. Indeed, by Lemma~\ref{lemma:no-wheels-in-beta}, $\Hub(\beta_M) \subseteq \{v_m, v_{m + 1}, \dots, v_\ell\}$. In particular, for each $x \in M$, $N_{\beta_M}(x) \cap \Hub(\beta_M)$ is the union of three cliques. Moreover, by Lemma~\ref{lemma:small-nbrs}, $\cchi(N_{\beta_M}(x) \setminus \Hub(\beta_M)) \leq 2$, and consequently $N_{\beta_M}(x)$ can be covered by 5 cliques as claimed. Now assume that $x \in \{v_1, \dots, v_{m - 1}\} \setminus M$. In this case, we may find $y \in M$ with $y <_A x$. But then $x \in A_y$, and $N_{\beta_M}(x) \subseteq \beta_M \cap C_y \subseteq \{y\} \cup N_{\beta_M}(y)$, which can be covered by 5 cliques by the previous argument. This yields (\ref{alpha-4}). 

\bigskip

Consider now the bipartite graph $H = (A, B)$ where $A = \{a_1, \dots, a_r\}$, $B = \{b_1, \dots, b_s\}$, each $a_i$ corresponds to a component $D_i$ of $G \setminus (\beta_M \cup \tilde{X})$, and each $b_i$ corresponds to a component $Q_i$ of $\beta_M \setminus \tilde{X}$, with an edge between $a_i$ and $b_j$ if there is an edge between $D_i$ and $Q_j$. Recall that each $D_i$ belongs to $A^*(f(v))$ for a unique $v \in \{v_1, \dots, v_{m - 1}\}$, namely its anchor. Write $v(i)$ to denote the anchor of $D_i$. Let $H_{\Core}$ be the subgraph of $H$ induced by $B \cup \{a_i \in A: v(i) \in \tilde{X} \cap \beta_M\}$. Define $\gamma := 8t + 3$. We claim that:

\sta{\label{no-long-hole} $H_{\Core}$ contains no hole of length at least $\gamma$.}

  Suppose for a contradiction that $H_{\Core}$ contains such a hole $L$. Without loss of generality, the vertices of $L$ are $a_1 \dd b_1 \dd a_2 \dd b_2 \dd \dots \dd a_{\gamma'} \dd b_{\gamma'} \dd a_1$ for some $\gamma' \geq \lceil\gamma/2\rceil > 4t$. Let $f_1, f_2, \dots, f_{2\gamma'}$ be the edges of $L$, directed along the hole, with $f_1 = (a_1, b_1)$. 

We may find a hole $W = x_1 \dd y_1 \dd P_1 \dd x_2 \dd y_2 \dd R_1 \dd x_3 \dd y_3 \dd P_2 \dd x_4 \dd y_4 \dd R_2 \dd \dots \dd P_{\gamma'} \dd x_{2\gamma'} \dd y_{2\gamma'} \dd R_{\gamma'} \dd x_1$ in $G \setminus \tilde{X}$ with the following properties: 

\begin{itemize}
    \item For each $i$ with $1 \leq i \leq 2\gamma'$, $x_i$, respectively $y_i$, belongs to the component corresponding to the tail, respectively head of $f_i$. 

    \item For each $i$ with $1 \leq i \leq \gamma'$, $P_i$ is a (possibly one-vertex) path in $Q_i$ between $y_{2i - 1}$ and $x_{2i}$. 

    \item For each $i$ with $1 \leq i \leq \gamma'$, $R_i$ is a (possibly one-vertex) path in $D_i$ between $y_{2i}$ and $x_{2i + 1}$ (indices modulo $2\gamma'$). 
\end{itemize}

\smallskip

Since $\gamma' > 4t$, the pigeonhole principle yields five distinct indices $i_1, \dots, i_{5}$ such that the anchors $v(i_1), \dots, v(i_{5})$ belong to the same clique, say $\tilde{K_1}$. It might be that some of those anchors coincide; however, we claim that at least two of them are distinct.
Indeed, note that, by (\ref{alpha-4}), a fixed $v \in \tilde{X} \cap M$ can have neighbors in at most five of the $Q_i$, since $\cchi(N_{\beta_M}(v)) \leq 5$. If all five anchors coincide, then we consider $N_L(\{a_{i_1}, \dots, a_{i_5}\})$ (recall that $L$ is the hole in $H_{\Core}$ that we started with). Since $L$ has size more than 10, this neighborhood consists of at least 6 nodes corresponding to components of $\beta_M \setminus \tilde{X}$. But then $v(i_1) (= v(i_2) = \dots = v(i_5))$ has neighbors in all 6 of those components, which is not possible. Thus we may assume that two of the anchors, say $v(i_1)$ and $v(i_2)$ are distinct.

We note that each of $v(i_1)$ and $v(i_2)$ has two non-adjacent neighbors on $W$. Indeed, this is the case by construction, since $D_{i_1}$ and $D_{i_2}$ each have edges to a pair of distinct $Q_j$s, and $N_{\beta_M}(D_i) \subseteq N_{\beta_M}[v(i)]$. But $v(i_1)$ and $v(i_2)$ are adjacent (since they are distinct vertices of $\tilde{K_1}$), and they have no common neighbors in $W$ (since all common neighbors of the two vertices belong to $\tilde{X}$). This is a contradiction to Lemma~\ref{lemma:common_nbrs}, which proves (\ref{no-long-hole}).

\smallskip

We next claim that:

\sta{\label{bdd-deg} We have $\deg_{H_{\Core}}(a_i) \leq 5$ for every $a_i \in A \cap H_{\Core}$.}

and 

\sta{\label{deg-1} We have $\deg_H(a_i) \leq 1$ for each $a_i \in A \setminus H_{\Core}$.} 

Indeed (\ref{bdd-deg}) follows directly from (\ref{alpha-4}), since we have $N_{\beta_M}(D_i) \subseteq N_{\beta_M}[v(i)]$. To see \eqref{deg-1}, note that, by the definition of $H_{\Core}$, any vertex $a_i \in A \setminus H_{\Core}$ has its anchor $v(i)$ contained in some component $Q$ of $\beta_M \setminus \tilde{X}$. In particular, the anchor (and thus the component $D_i$ of $G \setminus (\beta_M \cup \tilde{X})$ corresponding to $a_i$) cannot have neighbors in any other component of $\beta_M \setminus \tilde{X}$, meaning $\deg_H(a_i) \leq 1$ in this case, as claimed.

\smallskip

In \cite{weissauer-local}, it is shown that the treewidth of graphs without a complete bipartite subgraph and without a large induced hole is bounded (where the bound only depends on the sizes of the forbidden bipartite subgraph, and of the largest induced hole). Hence (\ref{no-long-hole}) and (\ref{bdd-deg}) imply that $k := \tw(H_{\Core})$ is bounded above by some constant depending only on $t$. By (\ref{deg-1}), $\tw(H) = k$ (assuming without loss of generality that it is at least 1). Let $w'_H$ be the weight function defined on $H$ by letting $w'_H(a_i)$, respectively $w'_H(b_j)$, be the $w$-weight of $D_i$, respectively $Q_j$. Let $w_H$ be the normalisation of $w'_H$, so that $w_H(H) = 1$. Explicitly, $w_H(x) := w'_H(x)/w'_H(H)$ for each $x \in H$ (and we note that we may assume $w'_H(H) \neq 0$, since otherwise $w(\tilde{X}) = 1$, and $\tilde{X}$ is a $(w, \frac{1}{2})$-balanced separator whose clique cover number is bounded by $t$). From Lemma~\ref{lemma:tw-to-weighted-separator}, it follows $H$ has a $(w_H, \frac{1}{2})$-balanced separator of size at most $k + 1$, which we call $Z$. Put $Z' := (Z \cap B) \cup N_H(Z \cap A)$, and note that (\ref{bdd-deg}) and (\ref{deg-1}) immediately imply:

\sta{\label{size-of-Z'} $|Z'| \leq 5(k + 1)$.}

\smallskip

We now claim the following:

\sta{\label{bdd-b} For every fixed $j \in [s]$, we have $|\{i \in [m - 1] : v_i \in \tilde{X} \cap \beta_M \text{ and } N(v_i) \cap Q_j \neq \emptyset\}| \leq 2t + 1$.} 

To see this, assume for a contradiction that this is not the case. Then we may find distinct $i_1, i_2, i_3$ such that $v_{i_1}, v_{i_2}$ and $v_{i_3}$ all belong to the same clique in the clique partition of $\tilde{X}$ and, say, $N(v_{i_\alpha}) \cap Q_1 \neq \emptyset$ for each $\alpha \in [3]$. We now apply Lemma~\ref{lemma:three_vtx_attachments} to those three vertices and a minimal induced subgraph of $Q_1$ containing neighbors of all three. We note that $Q_1$ contains no common neighbors of any two of the three vertices (since any such common neighbors belong to $\tilde{X}$, by diamond-freeness). In view of this, the second and third outcomes of the lemma imply that $\beta_M$ contains a pyramid, respectively a prism, which is impossible. In the first outcome, one of $v_{i_1}, v_{i_2}, v_{i_3}$ is the center of a wheel in $\beta_M$ which is not a bug -- once more impossible, by Lemma~\ref{lemma:no-wheels-in-beta}. This yields our desired contradiction and proves \eqref{bdd-b}.

\bigskip

We are now ready to construct our $(w, \frac{1}{2})$-balanced separator for $G$. To do so, we let $$J := \{j \in [s]: b_j \in Z'\},$$ and $$I := \bigcup\limits_{j \in J} \{i \in [r] : v_i \in \tilde{X} \cap \beta_M \text{ and } N(v_i) \cap Q_j \neq \emptyset\}.$$ We then put $$Y := \tilde{X} \cup \bigcup\limits_{i \in I} N_{\beta_M}(v_i).$$

By (\ref{alpha-4}), (\ref{size-of-Z'}) and (\ref{bdd-b}), we know that $Y$ can be covered by at most $t + 5 \cdot (2t + 1) \cdot 5(k + 1)$ cliques ($t$ cliques for $\tilde{X}$, and 5 cliques for each of the neighborhoods, of which there are at most $|I| \leq (2t+1)|Z'|$). This yields our desired upper bound on $\cchi(Y)$ which depends only on $t$. It remains to check that $Y$ is, indeed, a $(w, \frac{1}{2})$-balanced separator of $G$. 

\bigskip

To do so, we first note that every component of $H \setminus Z'$ is either contained inside a component of $H \setminus Z$, or consists of a single vertex in $Z \cap A$. In particular, we have: 

\sta{\label{comps-of-h-z} For every component $F$ of $H \setminus Z'$, if $F^G$ is the union of all $D_i$ and $Q_j$ corresponding to vertices of $F$, then $w(F^G) \leq \frac{1}{2}$.} 

Indeed, this is true if $F$ consists of a single vertex in $Z \cap A$, since then $F^G$ belongs to $A(f(v))$ for some unbalanced $v$, and so $$w(F^G) \leq w(A(f(v)) \leq \frac{1}{2}.$$ It is also true if $F$ belongs to some component of $H \setminus Z$, since then $w_H(F) \leq \frac{1}{2}$, and we have $$w(F^G) = w'_H(F) \leq w_H(F) \leq \frac{1}{2}.$$ 

\bigskip

Now let $S$ be a component of $G \setminus Y$. In particular, $S \subseteq D_1 \cup \dots \cup D_r \cup Q_1 \cup \dots \cup Q_s$. Let $S_H := H[\{a_i : S \cap D_i \neq \emptyset\} \cup \{b_j : S \cap Q_j \neq \emptyset\}]$, and note that $S_H$ is a connected induced subgraph of $H$. Finally, we look at two cases:

\bigskip

\begin{itemize}
    \item $S_H \cap Z' = \emptyset$. In this case, $S_H$ is contained in a connected component of $H \setminus Z'$, and $S \subseteq (S_H)^G$ (where as above, $(S_H)^G$ is the union of the connected components corresponding to vertices of $S_H$), so $w(S) \leq \frac{1}{2}$ by (\ref{comps-of-h-z}).

    \bigskip

    \item $S_H \cap Z' \neq \emptyset$. In this case, we first note that $|S_H \cap Z'| = 1$. Indeed, let $b_j \in S_H \cap Z'$, and note that, by construction of $Y$, the only $D_i$ with a neighbor in $Q_j \setminus Y$ are $D_i$ with $v(i) \in Q_j$. These $D_i$ have no neighbors in any other $Q_j$ (since $N(D_i) \subseteq N[v(i)]$). In particular, $S_H$ is a star centered at $b_j$, and $v(i) \in Q_j$ for all $a_i \in S_H$. Now

    $$w(S) \leq w((S_H)^G) \leq w_M(Q_j) \leq \frac{1}{2},$$
    as required.
\end{itemize}

\bigskip

This finishes the proof.
\end{proof}

Finally, we restate and prove Theorem~\ref{t1}. 
\setcounter{section}{1}
\setcounter{theorem}{1}
\begin{theorem}
\label{thm:C-bdd-tw}
There exists an integer $\Gamma$ such that $\ta(G) \leq \Gamma$ for every graph $G \in \mathcal C$. 
\end{theorem}
\begin{proof}
By Lemma \ref{lemma:clique-cutsets-tc}, we may assume that $G$ has no clique cutset, so $G \in \CC$. Let $w:V(G) \to [0, 1]$ be a weight function on $G$ with $w(G) = 1$, and let $(\{v_1, \hdots, v_\ell\}, m, M, \tilde{S}_M)$ be the hub division of $G$. Let $\beta_M$ be the central bag for $\tilde{S}_M$ and let $w_M$ be the inherited weight function on $\beta_M$. By Theorem \ref{thm:mainthm-betabs}, $\beta_M$ has a $(w_M, \frac{1}{2})$-balanced separator of clique cover number at most $9$. Now, by Theorem \ref{thm:extending-bs}, $G$ has a $(w, \frac{1}{2})$-balanced separator of clique cover number at most $\Gamma'(9)$. 
Finally, by Lemma \ref{lemma:bs-to-ta}, $\tc(G) \leq \Gamma := g(\Gamma'(9), \frac{1}{2})$, where $g$ is as defined in the lemma, and by Lemma~\ref{lem:ta-vs-tc}, $\ta(G) \leq \tc(G)$.
\end{proof}

\section*{Acknowledgement}

This research was in part performed during the Second 2022 Barbados Graph Theory Workshop at the McGill
University Bellairs Research Institute in Barbados, and the authors are grateful to the institute for
its facilities and hospitality. 


\end{document}